\documentclass[12pt]{amsart}
\usepackage{graphicx,amssymb}
\usepackage[usenames]{color}
\usepackage{amsthm,amsfonts,amsmath,amssymb,latexsym,epsfig,mathrsfs,yfonts,marvosym,latexsym,epsfig}
\usepackage[all]{xy}
\DeclareMathAlphabet\oldmathcal{OMS}        {cmsy}{b}{n}
\SetMathAlphabet    \oldmathcal{normal}{OMS}{cmsy}{m}{n}
\DeclareMathAlphabet\oldmathbcal{OMS}       {cmsy}{b}{n}
\usepackage{eucal}

\newtheorem{problem}{Problem}
\newtheorem{theorem}{Theorem}[section]
\newtheorem{lemma}[theorem]{Lemma}
\newtheorem{proposition}[theorem]{Proposition}
\newtheorem{corollary}[theorem]{Corollary}

\newtheorem{def/prop}[theorem]{Definition/Proposition}

\theoremstyle{definition}
\newtheorem{definition}[theorem]{Definition}
\newtheorem{remark}[theorem]{Remark}
\newtheorem*{ack}{Acknowledgements}
\newtheorem{example}{Example}[section]

\DeclareSymbolFont{bbold}{U}{bbold}{m}{n}
\DeclareSymbolFontAlphabet{\mathbbold}{bbold}

\def\BOne{\mathchoice{\scalebox{1.16}{$\displaystyle\mathbbold 1$}}{\scalebox{1.16}{$\textstyle\mathbbold 1$}}{\scalebox{1.16}{$\scriptstyle\mathbbold 1$}}{\scalebox{1.16}{$\scriptscriptstyle\mathbbold 1$}}}
\def\fract#1#2{\raise4pt\hbox{$ #1 \atop #2 $}}

\def\bbc{{\mathbb C}}

\def\bbr{{\mathbb R}}
\def\bbs{{\mathbb S}}
\def\bbt{{\mathbb T}}
\def\bbu{{\mathbb U}}

\def\gri{\iota}
\def\grk{\kappa}

\def\grr{\rho}

\def\grz{\zeta}

\def\grG{\Gamma}

\def\grS{\Sigma}

\def\bfl{{\bf l}}

\def\bfs{{\bf s}}

\def\bfw{{\bf w}}

\def\bfH{{\bf H}}
\def\bfS{{\bf S}}
\def\bfF{{\bf F}}
\def\bfV{{\bf V}}

\def\cald{{\mathcal D}}
\def\cale{{\mathcal E}}
\def\calf{{\mathcal F}}
\def\calg{{\mathcal G}}
\def\calh{{\mathcal H}}

\def\calk{{\mathcal K}}

\def\cals{{\oldmathcal S}}
\def\calS{{\mathcal S}}

\def\calw{{\mathcal W}}

\def\ga{{\mathfrak a}}

\def\gc{{\mathfrak c}}

\def\ge{{\mathfrak e}}

\def\gh{{\mathfrak h}}

\def\gn{{\mathfrak n}}
\def\go{{\mathfrak o}}

\def\gr{{\mathfrak r}}

\def\gt{{\mathfrak t}}
\def\gu{{\mathfrak u}}

\def\gA{{\mathfrak A}}

\def\gC{{\mathfrak C}}

\def\gR{{\mathfrak R}}

\def\<{\langle}
\def\>{\rangle}
\def\ra#1{\to}

\def\grad{\rm grad}
\def\crit{\rm crit}

\def\fract#1#2{\raise4pt\hbox{$ #1 \atop #2 $}}

\def\hook{\mathbin{\hbox to 6pt{%
                 \vrule height0.4pt width5pt depth0pt
                 \kern-.4pt
                 \vrule height6pt width0.4pt depth0pt\hss}}}

\begin{document}

\title{Some Open Problems in Sasaki Geometry} 
\author{Charles P. Boyer}
\author{ Hongnian Huang}
\author{Eveline Legendre}
\author{Christina W. T{\o}nnesen-Friedman}

\date{\today}
\address{Charles P. Boyer, Department of Mathematics and Statistics,
University of New Mexico, Albuquerque, NM 87131.}
\email{cboyer@unm.edu} 
\address{Hongnian Huang, Department of Mathematics and Statistics,
University of New Mexico, Albuquerque, NM 87131.}
\email{hnhuang@gmail.com} 
\address{Eveline Legendre\\ Universit\'e Paul Sabatier\\
Institut de Math\'ematiques de Toulouse\\ 118 route de Narbonne\\
31062 Toulouse\\ France}
\email{eveline.legendre@math.univ-toulouse.fr}
\address{Christina W. T{\o}nnesen-Friedman, Department of Mathematics, Union
College, Schenectady, New York 12308, USA } \email{tonnesec@union.edu}

\thanks{The first author was partially supported by grant \#519432 from the Simons Foundation. The third author is partially supported by France ANR project EMARKS No ANR-14-CE25-0010. The fourth author was partially supported by grant \#422410 from the Simons Foundation.}

\maketitle

\markboth{Problems in Sasaki Geometry}{C. Boyer, H. Huang, E. Legendre and C. T{\o}nnesen-Friedman}

\section{Introduction}
The purpose of this paper is to discuss two open problems in Sasaki geometry. These problems involve the so-called Sasaki cone which, although different in nature, plays a role in Sasaki geometry similar to that of the K\"ahler cone in K\"ahler geometry.  In the latter it is well understood by simple examples that constant scalar curvature (cscK) metrics need not be isolated and that there are complex manifolds whose K\"ahler cone admits extremal K\"ahler metrics, but no cscK metric. However, in the case of the Sasaki cone there are no known examples of the analogous phenomenon. This leads to two important open problems:

\begin{problem}\label{prob1}
Are constant scalar curvature rays in the Sasaki cone isolated?
\end{problem}

\begin{problem}\label{prob2}
If there are extremal rays of Sasaki metrics in the Sasaki cone, is there always at least one constant scalar curvature ray?
\end{problem}

When the contact bundle $\cald$ has vanishing first class (or more, generally is a torsion class), cscS metrics turn out to be $\eta$--Sasaki--Einstein and up to a transversal homothety, Sasaki--Einstein. In this case, by a famous result of Martelli, Sparks and Yau \cite{MaSpYau06}, the constant scalar curvature ray, whenever it exists, is not only isolated but unique in the Sasaki cone. This fact has been used recently by Donaldson and Sun in \cite{DonSun14,DonSun17} to study the moduli space of compact K\"ahler Fano manifolds and more precisely to prove uniqueness of the rescaled pointed Gromov-Hausdorff limits in this setting. An affirmative answer to Problem \ref{prob1} would be an extension of this very useful Martelli-Sparks-Yau Theorem.  

A partial answer to Problem \ref{prob1} was given by Lemma 4.1 in \cite{BHLT15} as well as by Corollary 1.7 in \cite{BHL17}. In particular the latter says that if the zero set $Z$ of rays of the Sasaki-Futaki invariant lies in a 2-dimensional subcone of the Sasaki cone $\gt^+(\cald,J)$, it is a finite set. Moreover, for toric contact structures on a lens space bundle over $S^2$ the Sasaki cone $\gt^+(\cald,J)$ has dimension $3$ and in \cite{Leg10} it is proved that all constant scalar curvature rays (cscS) are isolated in this case. In Theorem \ref{isothm} below we give another partial result. 

In the case of Problem \ref{prob2} involving the so-called $S^3_\bfw$ join $M\star_\bfl S^3_\bfw$ where $M$ is a Sasaki manifold with constant scalar curvature, it was proven in \cite{BoTo14a} that $M\star_\bfl S^3_\bfw$ has a cscS Sasaki metric. However, this uses the admissible construction of Apostolov, Calderbank, Gauduchon, and T{\o}nnesen-Friedman \cite{ApCaGa06,ACGT04,ACGT08} for which $M$ itself needs to be cscS. On the other hand when the contact bundle $\cald$ has vanishing first class (or more generally is a torsion class), there are known obstructions to the existence of cscS metrics due to Gauntlett, Martelli, Sparks, and Yau \cite{GMSY06}. Moreover, it was shown in \cite{BovCo16} that in all these cases the Sasaki cone admits no extremal metrics whatsoever. In terms of K-stability an affirmative answer to Problem \ref{prob2} is equivalent to stating that any Sasaki cone that admits a K-semistable polarization relative to a fixed maximal torus $\bbt$ also admits a K-semistable polarization with respect to an arbitrary torus (for notions of K-stability in the Sasaki context, cf. \cite{CoSz12,BHLT15,BovCo16}).

The Sasaki cone $\grk(\cald,J)$ can be thought of as the moduli space of Sasaki metrics with a fixed underlying contact CR structure $(\cald,J)$ where $\cald$ is the contact bundle and $J$ is a complex structure on $\cald$. On a Sasaki manifold $M^{2n+1}$ the dimension $k$ of $\grk(\cald,J)$ satisfies $1\leq k\leq n+1$. We are also interested in the moduli space $\ge(\cald,J)$ of extremal Sasaki metrics as well as the moduli space $\grk_{csc}(\cald,J)$ of constant scalar curvature Sasaki metrics. A result of \cite{BGS06} says that $\dim \ge(\cald,J)$ is either $0$ or $k$, and an affirmative answer to Problems \ref{prob1} and \ref{prob2} says that if  $\dim \ge(\cald,J)=k$, the dimension of $\grk_{csc}(\cald,J)$ is exactly $1$. The full moduli space, which is described in \cite{Boy18}, is obtained by varying $J$ also. However, in this note we fix the contact CR structure $(\cald,J)$ to study two important functionals on $\grk(\cald,J)$, the Einstein-Hilbert functional $\bfH$ and the Sasaki energy functional $\cals\cale$.  The variational calculus for $\bfH$ was performed in \cite{BHLT15},  so in this note we derive the Euler-Lagrange equations for $\cals\cale$ and compare the critical sets of $\bfH$ and $\cals\cale$. We end with an application to the case when $M$ is a lens space bundle over a compact Riemann surface.

\begin{ack}
This paper is roughly based on a talk given by the first author at the Australian-German Workshop on Differential Geometry in the Large held at the mathematical research institute MATRIX in Creswick, Victoria, Australia, Feb.2-Feb.14, 2019. He would like to thank MATRIX for its hospitality and support.
\end{ack}

\section{Brief Review of Sasaki Geometry}
Recall that a Sasakian structure on a contact manifold $M^{2n+1}$ of dimension $2n+1$ is a special type of contact metric structure $\cals=(\xi,\eta,\Phi,g)$ with underlying almost CR structure $(\cald,J)$ where $\eta$ is a contact form such that $\cald=\ker\eta$, $\xi$ is its Reeb vector field, $J=\Phi |_\cald$, and $g=d\eta\circ (\BOne \times\Phi) +\eta\otimes\eta$ is a Riemannian metric. $\cals$ is a Sasakian structure if $\xi$ is a Killing vector field and the almost CR structure is integrable, i.e. $(\cald,J)$ is a CR structure. We refer to \cite{BG05} for the fundamentals of Sasaki geometry. We call $(\cald,J)$ a {\it CR structure of Sasaki type}, and $\cald$ a {\it contact structure of Sasaki type}. We shall always assume that the Sasaki manifold $M^{2n+1}$ is compact and connected.

\subsection{The Sasaki Cone}
Within a fixed contact CR structure $(\cald,J)$ there is a conical family of Sasakian structures known as the Sasaki cone. We are also interested in a variation within this family. To describe the Sasaki cone we fix a Sasakian structure $\cals_o=(\xi_0,\eta_o,\Phi_o,g_o)$ on $M$ whose underlying CR structure is $(\cald,J)$ and let $\gt$ denote the Lie algebra of the maximal torus in the automorphism group of $\cals_o$. The {\it (unreduced) Sasaki cone} \cite{BGS06} is defined by
\begin{equation}\label{sascone}
\gt^+(\cald,J)=\{\xi\in\gt~|~\eta_o(\xi)>0~\text{everywhere on $M$}\},
\end{equation}
which is a cone of dimension $k\geq 1$ in $\gt$. The reduced Sasaki cone $\grk(\cald,J)$ is $\gt^+(\cald,J)/\calw$ where $\calw$ is the Weyl group of the maximal compact subgroup of $\gC\gR(\cald,J)$ which, as mentioned previously, is the moduli space of Sasakian structures with underlying CR structure $(\cald,J)$. However, it is more convenient to work with the unreduced Sasaki cone $\gt^+(\cald,J)$. It is also clear from the definition that $\gt^+(\cald,J)$ is a cone under the transverse scaling defined by
\begin{equation}\label{transscale}
\cals=(\xi,\eta,\Phi,g)\mapsto \cals_a=(a^{-1}\xi,a\eta,g_a),\quad g_a=ag+(a^2-a)\eta\otimes\eta, \quad a\in\bbr^+
\end{equation}
So Sasakian structures in $\gt^+(\cald,J)$ come in rays, and since the Reeb vector field $\xi$ is Killing $\dim\gt^+(\cald,J)\geq 1$, and it follows from contact geometry that $\dim\gt^+(\cald,J)\leq n+1$. When $\dim\gt^+(\cald,J)=n+1$ we have a toric contact manifold of Reeb type studied in \cite{BM93,BG00b,Ler02a,Ler04,Leg10,Leg16}. In this case there is a strong connection between the geometry and topology of $(M,\cals)$ and the combinatorics of $\gt^+(\cald,J)$. Much can also be said in the complexity 1 case ($\dim\gt^+(\cald,J)=n$) \cite{AlHa06}.

We let $\gR(\cald,J)$ denote the set of rays in $\gt^+(\cald,J)$, so that $\gt^+(\cald,J)$ is the open cone over the semi-algebraic set $\gR(\cald,J)$ \cite{BoCoRo98}. The combinatorial structure of $\gR(\cald,J)$ can be involved. For example let $M$ be a toric Sasaki manifold that is an $S^1$ bundle over a compact toric Hodge manifold $N$ (or orbifold). A choice of Reeb vector field in $\gt^+(\cald,J)$  gives the intersection of the dual moment cone with a hyperplane giving a generalized Delzant polytope\footnote{We use the term generalized here since it is a Delzant polytope only if the Reeb field is regular and the polytope lies on an integral lattice with primitive normal vectors in which case the quotient is a Hodge manifold. If $P$ lies on the lattice but the normal vectors are not primitive, the Sasakian structure is quasi-regular and the quotient is a Hodge orbifold, and if $P$ does not lie on a integral lattice, the Sasakian structure is irregular and there is no well-defined quotient.} $P$. We think of the interior of the dual polytope $P^*$ as representing the space of rays $\gR(\cald,J)$. For example in the case that $N$ is a Bott manifold \cite{BoCaTo17} of complex dimension $n$, the closure $\overline{\gR(\cald,J)}$ is a {\it cross-polytope} or {\it n-cross} which is dual to $P$ which in this case is combinatorically an n-cube. A 3-cross is an octahedron.

\subsection{The Transverse Holomorphic Structure} 
A Sasakian structure $\cals=(\xi,\eta,\Phi,g)$ not only determines a CR contact structure $(\cald,J)$, but also a transverse holomorphic structure $(TM/\calf_\xi,\bar{J})$ where $\calf_\xi$ is the foliation defined by the Reeb vector field $\xi$. Here instead of fixing a contact structure $\cald$ we fix the Reeb vector field $\xi$. This gives a contractible space of Sasakian structures, viz
\begin{equation}\label{Sasspace}
{\mathcal S}(\xi,\bar{J})=\{\varphi\in C^\infty_B(M)~|~(\eta+d^c_B\varphi)\wedge(d\eta +i\partial_B\bar{\partial}_B\varphi)^n\neq 0,\quad \int_M\varphi\eta\wedge(d\eta)^n=0\},
\end{equation}
where $d^c_B=\frac{i}{2}(\bar{\partial}-\partial)$. The space ${\mathcal S}(\xi,\bar{J})$ is an infinite dimensional Frech\'et manifold. Each Reeb vector field gives an isotopy class ${\mathcal S}(\xi,\bar{J})$ of contact structures, and deter-
mines a basic cohomology class $[d\eta]_B\in H^{1,1}(\calf_\xi)$, and each representative determines a
transverse K\"ahler structure with transverse K\"ahler metric $g^T=d\eta\circ (\BOne\times \Phi)$. Note that $d\eta$ is not exact as a basic cohomology class, since $\eta$ is not a basic 1-form. We want to search for a `preferred' Sasakian structure $\cals_\varphi$ which represents the cohomology class $[d\eta]_B$. This leads to the study \cite{BGS06} of the Calabi functional given by Equation \eqref{calabifunct} below.

\subsection{The Lie Algebra of Killing Potentials}\label{KPsec}
Note that for a CR structure of Sasaki type the group $\gC\gR(\cald,J)$ of CR transformations has dimension
at least one. Moreover, if $M$ is compact $\gC\gR(\cald,J)$ is a compact Lie group except for
the standard CR structure on the sphere $\bbs^{2n+1}$ where $\gC\gR(\cald,J)=\bbs\bbu(n+1,1)$. We are mainly concerned with reducing things to a maximal torus $\bbt^k$ in $\gC\gR(\cald,J)$ where $1\leq k\leq n+1$, and its Lie algebra $\gt$. Before doing so we briefly discuss the holomorphic viewpoint which gives rise to an infinite dimensional Lie algebra $\gh^T(\xi,\bar{J})$ of transverse holomorphic vector fields; however, the infinite dimensional part of $\gh^T(\xi,\bar{J})$ is generated by the smooth sections $\grG(\xi)$ of line bundle generated by the Reeb vector field, so we have a finite dimensional quotient algebra $\gh^T(\xi,\bar{J})/\grG(\xi)$ whose complexification consists of the complexification of the Lie algebra $\ga\gu\gt(\cals)$ of the Sasaki automorphism group together with a possible non-reductive part. See \cite{BGS06} for details. Here we concern ourselves with the Abelian Lie algebras $\gt^\bbc=\gt\otimes\bbc$ and $\gt$ associated to the maximal torus action. We refer to the potentials associated to $\gt^\bbc$ as {\it holomorphy potentials} since the action of elements of $\gt^\bbc$ is transversely holomorphic. 

Consider the strict contact moment map $\mu:M\ra{1.8} \gc\go\gn(M,\eta)^*$ with respect to the Fr\'echet Lie group of strict contact transformations $\gC\go\gn(M,\eta)$ defined by
\begin{equation}\label{etamommap}
\langle\mu(x),X\rangle=\eta(X).
\end{equation} 
The function $\eta(X)$ is basic with respect to $\xi$ and there is a Lie algebra isomorphism between $\gc\go\gn(\eta)$ and the Lie algebra of smooth $\xi$ invariant functions $C^\infty(M)^\xi$ with Lie algebra structure given by the Jacobi-Poisson bracket defined by
\begin{equation}\label{JP}
\{f,g\}=\eta([{\rm X^f},{\rm X^g}])
\end{equation}
where $d(\eta(X^f))=-X^f\hook d\eta$. This isomorphism extends to  a Lie algebra isomorphism between $\gc\go\gn(\cald)$ and all smooth functions $C^\infty(M)$. Each choice of contact form $\eta$ or equivalently Reeb vector field $\xi$ defines such an isomorphism. We have

\begin{lemma}\label{infcont}
Any $X\in\gc\go\gn(\eta)$ can be written uniquely\footnote{The sign $\pm \Phi\circ\grad_T~\pi_gs_g^T$ in the formula depends on the sign convention for the transverse K\"ahler form $d\eta$. We use the convention  $g^T=d\eta\circ (\BOne\times \Phi)$ which gives a plus sign.} as
$$X=\Phi\circ\grad_T~\eta(X)+\eta(X)\xi$$
where the gradient is taken with respect to the transverse metric $g^T$. 
\end{lemma}

\begin{proof}
This follows from
$$0=\pounds_{X}\eta=d(\eta(X))+\Phi\circ \grad_T~\eta(X)\hook d\eta.$$
\end{proof}

We now restrict attention to the finite dimensional Lie subalgebra $\ga\gu\gt(\cals)$ of the compact Lie group $\gA\gu\gt(\cals)$ of Sasaki automorphisms. These are Killing vector fields in $\gc\go\gn(\eta)$ that commute with $\xi$ and leave $\Phi$ invariant. They also leave invariant the transverse K\"ahler structure. Following the K\"ahler case (cf. \cite{Gau09b}) we say that $\eta(X)$ is a {\it Killing potential} (for $\cals$) when $X\in\ga\gu\gt(\cals)$. We denote\footnote{We only consider those Sasakian structures in the family given by $\gt^+(\cald,J)$, so a choice of $\cals$ is equivalent to a choice of Reeb vector field $\xi\in\gt^+(\cald,J)$ which is equivalent to specifying the Sasaki metric $g$ of $\cals$. We often abuse notation and label objects by $\xi,\eta,g$ or $\cals$ depending on the emphasis.} by ${\calk}^\xi$ the real vector space spanned by all the Killing potentials. ${\calk}^\xi$ forms a Lie subalgebra of $C^\infty(M)^\xi$ isomorphic to the Lie algebra $\ga\gu\gt(\cals)$. 
We actually restrict further by considering a real maximal torus $\bbt^k$ of $\gA\gu\gt(\cals)$, and its real Lie algebra $\gt$. We let $\calh^\xi$ denote the subalgebra of $\calk^\xi$ that is isomorphic to $\gt$. Choosing a maximal torus in $\gA\gu\gt(\cals)$ is equivalent to choosing maximal torus in $\gC\gR(\cald,J)$ \cite{BGS06}, so the Lie algebra $\gt$ is independent of the choice of $\cals\in\gt^+(\cald,J)$; however, we emphasize that $\calh^\xi$ depends on $\xi\in\gt^+(\cald,J)$, since the isomorphism $\grz\mapsto \eta(\grz)\in \calh^\xi$ does. We are interested in how $\calh^\xi$ changes as $\xi$ varies in $\gt^+(\cald,J)$, so we define 
\begin{equation}\label{Killpotlgt}
\calh=\bigcup_{\xi\in\gt^+(\cald,J)}\calh^\xi.
\end{equation}
Since any $\xi$ is in the center of $\gc\go\gn(\eta)$ we have
$$\bigcap_{\{\xi\in\gt^+(\cald,J)\}}\gc\go\gn(\eta)=\gt.$$

\begin{lemma}\label{Killpotl}
Let $(M,\cald,J,\bbt^k)$ be a contact CR manifold of Sasaki type with an effective action of a torus $\bbt^k$ that preserves the CR structure. Then
\begin{enumerate}
\item the elements of $\calh$ are basic with respect to every $\xi\in\gt^+(\cald,J)$. Equivalently, $\calh\subset C^\infty(M)^\bbt$;
\item the nonconstant elements of $\calh^{\xi_1}$ and $\calh^{\xi_2}$ are related by 
$$\eta_2(\grz)=\frac{1}{\eta_1(\xi_2)}\eta_1(\grz);$$
giving an isomorphism $\calh^{\xi_1}\approx\calh^{\xi_2}$ of Abelian Lie algebras;
\item if $\xi_1,\xi_2\in\gt^+(\cald,J)$ are not colinear, then $\calh^{\xi_1}\cap\calh^{\xi_2}=\bbr$ where $\bbr$ denotes the constants;
\item each choice of $\xi\in\gt^+(\cald,J)$ gives a Lie algebra monomorphism $\gri_\xi:\bbr\ra{2.5} \calh^\xi$
defined by $\gri_\xi(a)=a\eta(\xi)=a$.
\end{enumerate}
\end{lemma}

\begin{proof}
Item (1) is well known. 
Given $\xi_1,\xi_2\in\gt^+(\cald,J)$ the corresponding contact forms $\eta_1$ and $\eta_2$ satisfy $\eta_2=f\eta_1$ for some nowhere vanishing smooth function $f$. But since $\xi_2$ is the Reeb field of $\eta_2$ this implies $f=\frac{1}{\eta_1(\xi_2)}$ proving the first part of (2). That we have a Lie algebra isomorphism follows from the Abelian nature of $\calh^\xi$.

To prove (3) for any $a\in\bbr$ take $\grz=a\xi_1$ and $\grz'=a\xi_2$ then $\eta_1(\grz)=a=\eta_2(\grz')$. So $\calh^\xi\cap\calh^{\xi'}$ contains the constants. Conversely, if $\eta_1(\grz)=\eta_2(\grz')$ for some $\grz,\grz'\in\gt$, then (2) implies 
$$\eta_1(\grz)=\frac{\eta_1(\grz')}{\eta_1(\xi_2)}.$$
But then the only way that $\eta_1(\grz)$ and $\eta_1(\grz')$ can both be in $\calh^{\xi_1}$ is that $\grz'=a\xi_2$ and $\grz=a\xi_1$ which implies (3). Item (4) is clear.
\end{proof}

We call the element $\eta_1(\xi_2)\in \calh^{\xi_1}$ the {\it transfer function} from $\eta_1$ to $\eta_2$. Note that the smooth functions $\eta_1(\xi_2)$ and $\eta_2(\xi_1)$ are invariant under the same transverse scaling $(\xi_1,\xi_2)\mapsto (a^{-1}\xi_1,a^{-1}\xi_2)$ and satisfy the following relations:
\begin{equation}\label{multeqn} 
\eta_1(\xi_1)=1,\qquad \eta_2(\xi_2)=1,\qquad \eta_2(\xi_1)\eta_1(\xi_2)=1.
\end{equation}

\section{Extremal Sasaki Geometry}
The notion of extremal K\"ahler metrics was introduced as a variational problem by Calabi in \cite{Cal56} and studied in greater depth in \cite{Cal82}. This was then emulated in \cite{BGS06} for the Sasaki case, namely  
\begin{equation}\label{calabifunct}
\cale_2(g)=\int_Ms^2_gdv_g
\end{equation}
where the variation is taken over the space $\calS(\xi,\bar{J})$. As in the K\"ahler case the Euler-Lagrange equation is a 4th order PDE 
\begin{equation}\label{ELeqn2}
\mathscr{L}\varphi=(\bar{\partial}\partial^\#)^*\bar{\partial}\partial^\#\varphi =\frac{1}{4}\bigl(\Delta_B^2\varphi +4g(\grr^T,i\partial \bar{\partial}\varphi)+2(\partial s^T)\hook\partial^\#\varphi\bigr) =0
\end{equation}
whose critical points are those Sasaki metrics whose $(1,0)$ gradient $\partial^\#s_g$ of the scalar curvature $s_g$ is transversely holomorphic. Such Sasaki metrics (structures) are called {\it extremal}. An important special case are the Sasaki metrics of constant scalar curvature (cscS) in which case $\partial^\#s_g$ is the zero vector field. 

Since both the volume functional and the total transverse scalar curvature functional do not depend on the representative in $\calS(\xi,\bar{J})$, the functional \eqref{calabifunct} is essentially equivalent to the functional
\begin{equation}\label{calabifunct2}
\cale^T_2(g)=\int_M(s^T_g)^2dv_g.
\end{equation}
This latter functional has the advantage of behaving nicely under transverse scaling \eqref{transscale}. It is important to realize that if $\cals$ is extremal, so is $\cals_a$ for all $a\in\bbr^+$, and if $g$ has constant scalar curvature so does $g_a$ for all $a\in\bbr^+$. So extremal and cscS Sasakian structures come in rays.

\subsection{Transverse Futaki-Mabuchi}
The Sasaki version $\chi$ of the Futaki-Mabuchi vector field \cite{FuMa95} was introduced in \cite{BovCo16} and used to define the Sasaki version of K-relative stability. Following \cite{FuMa95} we consider the $L^2$ inner product $\langle\cdot,\cdot \rangle$ on the polarized Sasaki manifold $(M,\cals)$, or more generally the inner product on tensors, $p$-form, functions, etc.
\begin{equation}\label{L2g}
\langle \alpha,\beta\rangle = \int_M g(\alpha,\beta) dv_g.
\end{equation}
The $L^2$ inner product on functions induces an inner product on the Lie algebra $\gt$ that depends on the choice of Reeb vector field $\xi\in\gt^+(\cald,J)$,
\begin{equation}\label{FMinprod}
\langle \zeta,\zeta'\rangle_\xi = \langle \eta(\zeta),\eta(\zeta')\rangle =\int_M\eta(\grz)\eta(\grz')dv_g.
\end{equation}

\begin{remark}
In fact, the inner product \eqref{FMinprod} defines an inner product on the Fr\'echet Lie algebra $\gc\go\gn(\eta)$; however, we shall only make use of it on the Abelian subalgebra $\gt$.
\end{remark}


Then for each $\xi\in\gt^+$ the inner product \eqref{FMinprod} gives an orthogonal splitting
\begin{equation}\label{orthogsplitgt}
\gt=\bbr\xi\oplus \gt_0,
\end{equation}
and under the isomorphism $\gt\approx \calh^\xi$ we have orthogonal splittings
\begin{equation}\label{orthogsplitcalh}
\calh^\xi\approx \bbr\oplus \calh^\xi_0
\end{equation}
where $\bbr$ denotes the constants and 
$$\calh^\xi_0=\{\eta(\grz)~|~\int_M\eta(\grz)dv_\xi=0,~ \grz\in\gt\}.$$

Letting $\gh^T(\xi,\bar{J})$ denote the Lie algebra of transverse holomorphic vector fields on $(M,\cals)$, we
recall the Sasaki-Futaki invariant \cite{BGS06} (or transversal Futaki invariant) $\bfF_\xi:\gh^T(\xi,\bar{J})\ra{2.1} \bbc$ defined by 
$$\bfF_\xi(X)=\int X\psi_gdv_\xi$$
where the basic transverse Ricci potential $\psi_g$ satisfies $\grr^T=\grr^T_h+i\partial\bar{\partial}\psi_g$ and $\grr^T_h$ is the harmonic part of the transverse Ricci form $\grr^T$. By Proposition 5.1 of \cite{BGS06} $\bfF_\xi$ only depends on the class $\calS(\xi,\bar{J})$, and we know that $\bfF_\xi$ is degenerate on $\gh^T(\xi,\bar{J})$ since it vanishes on the infinite dimensional subalgebra of sections $\grG(L_\xi)$ of the line bundle $L_\xi$. So we restrict attention to the finite dimensional Lie algebra $\gh^T_\xi= \gh^T(\xi,\bar{J})/\grG(L_\xi)$. From this we get a map $\bfF:\gt^+\times \gh^T_\xi:\ra{2.1} \bbc$ defined by $\bfF(\xi,X)=\bfF_\xi(X)$. By Lemma 4.6 of \cite{BGS06} it follows that if $\grz\in\gt$ then $\Phi\grz\in\gh^T_\xi$
which gives the map $\bfF:\gt^+\times \gt\ra{2.1} \bbr$ defined by $\bfF(\xi,a)=\bfF_\xi(\Phi(a))$. So for each Sasakian structure $\cals\in\gt^+$ we have its Sasaki-Futaki invariant $\bfF_\xi\circ\Phi:\gt\ra{2.5} \bbr$ on $\gt$ defined by
\begin{equation}\label{Futinv}
\bfF_\xi\circ\Phi(\grz)=\int \Phi\grz\psi_gdv_\xi.
\end{equation}

\begin{definition}\label{FMvf}\cite{BovCo16}
We define the {\it Sasaki-Futaki-Mabuchi vector field} $\chi_\xi$ to be the dual of $\bfF_\xi\circ\Phi$ with respect to the inner product \eqref{FMinprod} on $\gt$, that is $\bfF_\xi\circ\Phi(\grz)=\langle\chi_\xi,\grz\rangle_\xi$.
\end{definition}

So the Sasaki-Futaki invariant becomes
\begin{equation}\label{Futinv2}
\bfF_\xi\circ\Phi(\grz)=\int_M \eta(\grz)\eta(\chi_\xi)dv_\xi.
\end{equation}
The fact that $\bfF_\xi\circ\Phi(\xi)=0$ implies 
\begin{equation}\label{chinorm}
\langle\xi,\chi_\xi\rangle_\xi=\int_M\eta(\chi_\xi)dv_\xi=0,
\end{equation}
or equivalently $\eta(\chi_\xi)\in\calh^\xi_0$ and $\chi_\xi\in\gt_0$. 

Consider the projection $\pi_g:C^\infty(M)^\xi\longrightarrow \calh^\xi$ onto the space $\calh^\xi$ of Killing potentials, or equivalently $\pi:\gc\go\gn(\eta)\longrightarrow \gt$. From the orthogonal decomposition \eqref{orthogsplitcalh} we see that the projection of the scalar curvature $s^T_g$ onto the constants $\bbr$ is just the average scalar curvature $\bar{\bfs}_\xi$ defined by $\bar{\bfs}_\xi=\frac{\bfS_\xi}{\bfV_\xi}$  where $\bfS_\xi$ is the total transverse scalar curvature of $\cals$ and $\bfV_\xi$ is its volume. As in \cite{FuMa95} we have
\begin{lemma}\label{Futchi}
The Sasaki-Futaki-Mabuchi vector field $\chi_\xi$ can be written uniquely as
$$\chi_\xi=\Phi\circ\grad_T~\pi_gs_g^T+(\pi_gs^T_g-\bar{\bfs}^T_g)\xi$$
where the gradient is taken with respect to the transverse metric $g^T$. Moreover, $\chi_\xi$ is independent of the choice of $\cals\in\calS(\xi,\bar{J})$.
\end{lemma}

\begin{proof}
By Definition \ref{FMvf} $\chi_\xi$  is the unique vector field in $\gt_0$ that is dual to the Sasaki- Futaki invariant. Since $\gt$ is a subalgebra of $\gc\go\gn(\eta)$ and $\chi_\xi\in\gt$, Lemma \ref{infcont} says that $\chi_\xi$ takes the form $\Phi\circ\grad_T~\eta(\chi_\xi)+\eta(\chi_\xi)\xi$. But from \cite{BGS06} $F_\xi\circ\Phi$ is the transverse Futaki invariant with respect to the transverse K\"ahler metric $g^T$. Thus, $\chi_\xi$ is just the transverse Futaki-Mabuchi vector field with respect to $g^T$ which implies that $\eta(\chi_\xi)=\pi_gs^T_g$ up to a constant. But then since $\chi_\xi\in\gt_0$ and $s^T_g-\pi_gs^T_g$ is orthogonal to the constants we have
$$0=\int_M\eta(\chi_\xi)dv_g=\int_M\pi_gs^T_gdv_g+c\int_Mdv_g=\int_M(s^T_g+c)dv_g=(\bar{\bfs}^T_g+c)\bfV_g$$
which gives the result. The last statement follows from Definition \ref{FMvf} and Proposition 5.1 of \cite{BGS06}.
\end{proof}

It is easy to obtain the relationship between extremality and Killing potentials. 

\begin{lemma}\label{extremlem}
On a Sasaki manifold the following are equivalent:
\begin{enumerate}
\item $\cals$ is extremal;
\item $s^T_g\in\calh^\xi$;
\item $\pi_gs^T_g=s^T_g$;
\item $\pi_gs^T_g-\bar{\bfs}^T_g \in\calh^\xi_0$.
\end{enumerate}
Moreover, $\chi_\xi=0$ if and only if $\pi_gs^T_g=\bar{\bfs}^T_g$ whether it is extremal or not.
\end{lemma}


As a corollary of Lemma \ref{Futchi} we have

\begin{proposition}\label{inprodisotopy}
On the Lie algebra $\gt$, the inner product $\langle\cdot,\cdot\rangle_\xi$ depends only on the isotopy class $(\calS,\bar{J})$.
\end{proposition}

In \cite{FuMa02} Futaki and Mabuchi generalized their bilinear form to a multilinear form. Accordingly, we can do the same, although we make no use of it. For each $\xi\in\gt^+(\cald,J)$ we define the symmetric multilinear form $\varPhi_\xi^l:{\rm sym}^l(\gt)\longrightarrow \bbr$ by
\begin{equation}\label{FMmulti}
\varPhi_\xi^l(\grz_1,\ldots,\grz_l)=\int_M\eta(\grz_1)\cdots \eta(\grz_l)dv_g.
\end{equation}
As in \cite{FuMa02} we have 

\begin{proposition}\label{FMmultiinv}
$\varPhi_\xi^k$ only depends on the isotopy class $\calS(\xi,\bar{J})$.
\end{proposition}

\subsection{The Einstein-Hilbert Functional}
We now consider the Einstein-Hilbert functional studied in \cite{Leg10,BHLT15,BHL17}
\begin{equation}\label{HE}
\bfH_\xi=\frac{\bfS_\xi^{n+1}}{\bfV_\xi^n}
\end{equation}
as a function on the Sasaki cone $\gt^+$. Since both $\bfV_\xi$ and $\bfS_\xi$ are independent of the choice of Sasakian structure in $\calS(\xi,\bar{J})$, the Einstein-Hilbert functional $\bfH_\xi$ only depends on the isotopy class of contact structure, and is, moreover, invariant under transverse scaling. We are interested in the set $\crit~\bfH\subset \gt^+$ of critical points of $\bfH$. Since $\bfH$ is invariant under transverse scaling we can restrict $\crit~\bfH$ to an appropriate slice $\grS\subset \gt^+$ if desired. It is easy to check the following behavior under transverse scaling.
\begin{lemma}\label{transscalelem}
The following relations hold under the transverse scaling operation $\xi\mapsto a^{-1}\xi$:
\begin{enumerate}
\item $s^T_{a^{-1}\xi}=a^{-1}s^T_\xi$;
\item $\bar{\bfs}^T_{a^{-1}\xi}=a^{-1}\bar{\bfs}^T_\xi$;
\item $\bfS_{a^{-1}\xi}=a^n\bfS_\xi$;
\item $\bfV_{a^{-1}\xi}=a^{n+1}\bfV_\xi$;
\item  $\bfH_{a^{-1}\xi}=\bfH_\xi$
\item $\bfF_{a^{-1}\xi}=a^{n+1}\bfF_\xi$;
\item $\chi_{a^{-1}\xi}=a^{-2}\chi_\xi$;
\item $\langle\grz,\grz'\rangle_{a^{-1}\xi}=a^{n+3}\langle\grz,\grz'\rangle_\xi$.
\end{enumerate}
\end{lemma}

By Lemma \ref{transscalelem} the zeroes and critical points of $\bfV_\xi,\bfS_\xi,\bfF_\xi$ come in rays. 
We let $Z^+$ denote the zero set in $\gt^+(\cald,J)$ of the Sasaki-Futaki invariant $\bfF$. Likewise we let $Z\subset \gR(\cald,J)$ denote the zero set of rays of the Sasaki-Futaki invariant $\bfF$. We denote by $\gr_\xi$ the ray through $\xi\in\gt^+(\cald,J)$. 
From Definition \ref{FMvf} and Lemma \ref{Futchi} we have
\begin{proposition}\label{Z+prop}
A Reeb vector field $\xi\in \gt^+(\cald,J)$ lies in $Z^+$ if and only if $\pi_gs^T_g=\bar{\bfs}^T_g$. Moreover, $\xi\in Z^+$ is extremal if and only if it is cscS.
\end{proposition}

We are interested in the underlying structure of $Z$ and $Z^+$.
\begin{proposition}\label{Zalgvar} 
$Z$ and $Z^+$ are real affine algebraic varieties, and $Z^+$ is the cone over $Z$. Hence, the number of connected components of $Z$ is finite, and $Z$ is a compact subset of $\gR(\cald,J)$.
\end{proposition}

\begin{proof}
From Lemma 3.1 of  \cite{BHLT15} we have
\begin{equation}\label{lem31}
d\bfH_\xi (a) = n(n+1)\bar{\bfs}_\xi^n \bfF_{\xi}(\Phi(a)). 
\end{equation}
We first consider the case $\bar{\bfs}_\xi^n\neq 0$ and let $\gr_{\xi_o}\in Z$ be its corresponding ray. Moreover, since $s^T_o\neq 0$ there is a neighborhood $U_o$ of $\xi_o$ such that  Equation \eqref{lem31} holds in $U_o$ with $\bfS_\xi\neq 0$ for all $\xi\in U_o$. But from \cite{BHL17} we know that both the total transverse scalar curvature $\bfS$ and the Einstein-Hilbert functional $\bfH$ are rational functions of $\xi$. It follows from \eqref{lem31} that the Sasaki-Futaki invariant $\bfF$ is a rational function of $\xi$ on $U_o$. So its zero set is a real algebraic variety \cite{BoCoRo98}.
Now consider the case $\bfS_{\xi_o}=0$. The second statement of Lemma 3.1 in \cite{BHLT15} says that when $\bfS_{\xi}=0$ the following holds 
$$d\bfS_\xi=n\bfF_\xi\circ\Phi.$$
Thus, $\gr_{\xi_o}\in Z$ if and only if $\gr_{\xi_o}$ is a critical ray of $\bfS_{\xi_o}$ which is a rational function of $\xi_o$. So the result follows as above. The compactness of $Z$ follows since $\bfH$ is a proper function \cite{BHL17} .
\end{proof}

\begin{remark}\label{sT0rem}
Note that if $s^T_{\xi_o}=0$ then it has $s^T_{a^{-1}\xi_o}=0$ along the entire ray $\gr_{\xi_o}$ and it is the unique ray in $\gt^+(\cald,J)$ with this property by \cite{BHL17}. 
\end{remark}
Since real algebraic varieties are CW complexes Proposition \ref{Zalgvar} implies 
\begin{corollary}\label{Zlocpath}
$Z$ is locally path connected. 
\end{corollary}

But we know that a ray $\gr_\xi\in Z$ need not be extremal. Identifying the tangent space of $\gt^+$ at $\xi\in\gt^+$ with the Lie algebra $\gt$ itself, we show here that the gradient vector field $\grad~\bfH$ viewed as an element of $\gt$ is proportional to the Futaki-Mabuchi vector field $\chi_\xi$. Specifically we have the following corollary of Lemma 3.1 in \cite{BHLT15}:
\begin{theorem}\label{gradHthm}
For each $\xi\in\gt^+$ the vector field $\grad~\bfH_\xi$ satisfies
\begin{enumerate}
\item $\grad~\bfH_\xi =n(n+1)\bar{\bfs}_\xi^n\chi_\xi,$
\item $\langle\grad~\bfH_\xi,\xi\rangle_\xi=0$;
\item $\xi$ is a critical point of $\bfH_\xi$ if and only if $\grad~\bfH_\xi=0$;
\item  if $\bfS_\xi\neq 0$ then $\xi$ is a critical point of $\bfH_\xi$ if and only if $\pi_gs^T_g=\bar{\bfs}^T_g$;
\item if $\bfS_\xi\neq 0$ then $\xi$ is a critical point of $\bfH_\xi$ if and only if $\chi_\xi=0$. Moreover, in this case $\chi_\xi$ is a rational function of $\xi$.
\end{enumerate}
\end{theorem}

\begin{proof}
Taking the dual of equation \eqref{lem31} with respect to the Futaki-Mabuchi inner product $\langle\cdot,\cdot\rangle_\xi$ on $\gt$ and using Definition \ref{FMvf} gives (1) from which (2) follows, and (3) follows by duality. The last two statements follows from (1) and the results of \cite{BHL17}.
\end{proof}

\subsection{The Sasaki Energy Functional}
In \cite{BGS07b} a functional, the $L^2$ norm of the projection $\pi_gs_g$, which provides a lower bound to the Calabi functional \eqref{calabifunct}, namely
\begin{equation}\label{strongextrem}
\int_Ms^2_gdv_\xi\geq \int_M (\pi_gs_g)^2dv_g=:\cals\cale_2(\xi)
\end{equation}
was studied. However, this functional does not behave well under transverse scaling which is desirable when varying in the Sasaki cone. Thus, we consider a related functional
\begin{equation}\label{calscaleT2}
\cals\cale^T_2(g)=\int_M (\pi_gs^T_g)^2dv_g
\end{equation}
which gives a lower bound to the transverse Calabi energy functional \eqref{calabifunct2}, namely
\begin{equation}\label{transcalabi}
\cale_2^T(g)=\int_M(s^T_g)^2dv_\xi\geq \int_M (\pi_gs^T_g)^2dv_g,
\end{equation}
as we vary through elements in the Sasaki cone $\gt^+(\cald,J)$ with a fixed volume. However, as with the Einstein-Hilbert functional it is convenient to normalize and consider
\begin{equation}\label{strongextr}
\cals\cale(\xi):=\cals\cale^T(\xi)=\frac{(\int_M (\pi_gs^T_g)^2dv_g)^{n+1}}{(\int_Mdv_g)^{n-1}}
\end{equation}
which is homogeneous with respect to transverse scaling, that is, $\cals\cale(a^{-1}\xi)=\cals\cale(\xi)$. So the critical points of $\cals\cale$ are manifestly rays $\gr_\xi\in \gR(\cald,J)$. Following \cite{Sim00} for the K\"ahler case, we call $\cals\cale(\xi)$ the {\it Sasaki Energy} functional.

From the fact that $\chi_\xi\in\calh^\xi_0$ we have the equality 
\begin{equation}\label{piseqn}
\int_M\pi_gs^T_gdv_g=\int_Ms^T_gdv_g=\bfS_\xi
\end{equation}
which suggests a strong relation between the Einstein-Hilbert functional \eqref{HE} and the Sasaki energy functional \eqref{strongextr}.

We now consider the variation of the functionals \eqref{strongextr} and \eqref{calscaleT2}.  Generally we could consider a path of contact forms 
\begin{equation}\label{genpath}
\eta_t = \frac{1}{\eta(\xi_t)}\eta + d^c \varphi_t \, 
\end{equation}
where $\varphi_t$ is a $\xi_t$-basic function; however, it is enough to take the variation to lie within a fixed contact CR structure $(\cald,J)$ by choosing $\varphi_t$ to be a constant. We are mainly interested in the scale invariant functional $\cals\cale$, although it is easier to work with the functional $\cals\cale^T_2$. However, in order to obtain rays as critical points of $\cals\cale^T_2$ we need to choose a slice that intersects each ray once. However, from the definition of $\cals\cale$, Equation \eqref{strongextr}, we have
\begin{equation}\label{calscale2eqn}
\frac{d\cals\cale(\xi_t)}{dt}=(n+1)\frac{\cals\cale_2^T(\xi_t)^n}{\bfV (\xi_t)^{n-1}} \frac{d\cals\cale_2^T(\xi_t)}{dt} -(n-1)\frac{\cals\cale_2^T(\xi_t)^{n+1}}{\bfV (\xi_t)^{n}}\frac{d\bfV (\xi_t)}{dt}.
\end{equation}
Thus, if we choose variations of $\cals\cale^T_2$ with fixed volume, the critical points of $\cals\cale$ and $\cals\cale^T_2$ are essentially the same. Indeed, we have

\begin{lemma}\label{fixvollem}
Under variations of fixed volume a critical point of $\cals\cale^T_2$ is a critical point of $\cals\cale$. Conversely, if $s^T_g$ is not identically zero, a critical point of $\cals\cale$ is a critical point of $\cals\cale^T_2$ under variations of fixed volume.
\end{lemma}

\begin{remark}\label{critptrem}
Note that 
\begin{equation}\label{fixvol}
\frac{d\bfV (\xi_t)}{dt}=-(n+1) \int_M\eta(\dot{\xi})dv_g,
\end{equation}
so the fixed volume constraint is realized by the equation
\begin{equation}\label{volfix}
\int_M\eta(\dot{\xi})dv_g= 0.
\end{equation}
\end{remark}

We shall often make use of the following
\begin{lemma}\label{projlem}
For any $f\in C^\infty(M)^\xi$, $\pi_gf$ is the unique element $A$ in $\calh^\xi$ such that 
$$\langle A,h\rangle= \langle f,h\rangle$$ 
for all $h\in \calh^\xi$.
\end{lemma}

Next we give the Euler-Lagrange equations for both functional $\cals\cale^T_2$ and $\cals\cale$. 
\begin{theorem}\label{varthm1} 
For $t\in (-\epsilon,\epsilon)$ let $\xi_t$ be a $C^1$ path of Reeb vector fields compatible with a fixed CR structure $(\cald,J)$, then we have
\begin{equation}\label{variationthm}
\frac{d\cals\cale(\xi_t)}{dt}|_{t=0}= (n+1)\frac{\cals\cale^T_2(\xi)^n}{\bfV(\xi)^{n-1}} \int_MF(\xi)\eta(\dot{\xi})dv_g,
\end{equation}
where $F(\xi)$ is given by
$$F(\xi)=-2ns^T_g\pi_gs^T_g-2(2n+1)\Delta(\pi_gs^T_g) +(n+1)(\pi_gs^T_g)^2 +(n-1)\frac{\cals\cale^T_2(\xi)}{\bfV(\xi)}.$$
So the Euler-Lagrange equations of $\cals\cale$ are 
\begin{equation}\label{ELeqn}
\pi_g\Bigl(-2ns^T_g\pi_gs^T_g-2(2n+1)\Delta(\pi_gs^T_g) +(n+1)(\pi_gs^T_g)^2+(n-1)\frac{\cals\cale^T_2(\xi)}{\bfV(\xi)}\Bigr)=0.
\end{equation}
\end{theorem}

\begin{remark}\label{nocscrem}
Note that the formula for the variation of $\cals\cale^T_2$ is given by \eqref{variationthm} without the multiplicative factor and the last term in $F(\xi)$, namely
\begin{equation}\label{variationthm2}
\frac{d\cals\cale_2^T(\xi_t)}{dt}|_{t=0}= \int_M\Bigl(-2ns^T_g\pi_gs^T_g-2(2n+1)\Delta(\pi_gs^T_g) +(n+1)(\pi_gs^T_g)^2\Bigr)\eta(\dot{\xi})dv_g.
\end{equation}
So the Euler-Lagrange equations for variations of $\cals\cale^T_2$ with fixed volume are 
\begin{equation}\label{ELeqnvol}
\pi_g\Bigl(-2ns^T_g\pi_gs^T_g-2(2n+1)\Delta(\pi_gs^T_g)+(n+1)(\pi_gs^T_g)^2 \Bigr)=c
\end{equation}
where $c$ is any constant.

In Examples \ref{sehex} and \ref{sehex2} below we give critical points of $\cals\cale$, hence, solutions of \eqref{ELeqn} that are not cscS. Moreover, those in Example \ref{sehex2} consist of one cscS ray and two extremal rays that are not cscS. So in this case we have two solutions of \eqref{ELeqn} such that $\pi_gs^T_g=s^T_g$ which are not constant.
\end{remark}

For the proof of Theorem \ref{varthm1} we first give some lemmas the first of which was given in \cite{BGS07b} as well as \cite{BHLT15}.
\begin{lemma}\label{scacurv}
For variations over a $C^1$ path of the form \eqref{genpath} with a fixed contact CR structure we have
\begin{equation} \label{pr2}
\dot{s^T_t} = -(2n+1)\Delta (\eta(\dot{\xi}))+s^T \eta(\dot{\xi})\end{equation}
\end{lemma}

Next we have
\begin{lemma}\label{varlem1} 
Consider a $C^1$ path of Reeb vector fields $\xi_t$, $t\in (-\epsilon,\epsilon)$, compatible with a fixed CR structure $(\cald,J)$, then for any $h\in {\calh}^\xi$ 
\begin{equation}\label{variationpis}
\langle \dot{\pi_gs^T_g}, h \rangle = -(2n+1) \langle \Delta^g \eta(\dot{\xi}), h \rangle - n \langle  s^T_g \eta(\dot{\xi}), h \rangle+ (n+1)\langle  \pi_gs^T_g \eta(\dot{\xi}), h\rangle. 
\end{equation}
 \end{lemma}

\begin{proof} 
Apply Lemma \ref{projlem} with $f=s^T_g$ and use the variation of the volume form 
\begin{equation}\label{varvol}
\Bigl(\frac{d}{dt}dv_t\Bigr)_{t=0}=-(n+1)\eta(\dot{\xi})dv_g
\end{equation}
to give
$$\langle \dot{\pi_gs^T_g}, h \rangle = \langle \dot{s^T_g}, h \rangle -(n+1)\langle s^T_g, h \eta(\dot{\xi})\rangle + (n+1)\langle  \pi_gs^T_g, h \eta(\dot{\xi})\rangle.$$ 
Applying Lemma \ref{scacurv} to this gives the claim.
\end{proof}

\begin{proof}[Proof of Theorem \ref{varthm1}]
Using \eqref{calscale2eqn} we obtain 

\begin{eqnarray}\label{ELeqngen}
\frac{d\cals\cale(\xi_t)}{dt}|_{t=0} &=& (n+1)\frac{\cals\cale^T_2(\xi)^n}{\bfV(\xi)^{n-1}}\int_M\bigl(2\pi_gs^T_g(\dot{\pi_gs^T_g}) -(n+1)(\pi_gs^T_g)^2\bigr)\eta(\dot{\xi})dv_g \notag \\
                &+&(n+1)(n-1)\frac{\cals\cale^T_2(\xi)^{n+1}}{\bfV(\xi)^n}\int_M\eta(\dot{\xi}). 
\end{eqnarray}
Putting $h=\pi_gs^T_g$ in \eqref{variationpis} the first integral becomes
$$
 -2(2n+1)\int_M(\pi_gs^T_g)\Delta \eta(\dot{\xi})dv_g -2n\int_Ms^T_g(\pi_gs^T_g) \eta(\dot{\xi})dv_g +(n+1)\int_M(\pi_gs^T_g)^2\eta({\dot\xi})dv_g.
$$
Integrating the first term by parts twice and rearranging \eqref{ELeqngen} gives \eqref{variationthm}. Then, since $\eta(\dot{\xi})$ is an arbitrary element of $\calh^\xi$, the Euler-Lagrange equations \eqref{ELeqn} follows from \eqref{variationthm} by applying Lemma \ref{projlem}.
\end{proof}


\begin{proposition}\label{critgrad}
Let $\cals$ be a Sasakian structure. Then 
\begin{enumerate}
\item any ray $\gr_\xi\in Z$ is a critical point of $\cals\cale$.
\item if $\bfS_\xi\neq 0$ a critical point of $\bfH$ is a critical point of $\cals\cale$.
\end{enumerate}
\end{proposition}

\begin{proof}
We easily see that $\pi_gs^T_g=\bar{\bfs}^T_g$ is a solution of \eqref{ELeqn} proving (1). Item (2) then follows from (4) of Theorem \ref{gradHthm}.
\end{proof}

\begin{remark}\label{Z+2rem}
By Proposition \ref{Z+prop} a Reeb field $\xi$ is in $Z^+$ if and only if $\pi_gs^T_g=\bar{\bfs}^T_g$, and we see that this is a solution to the Euler-Lagrange equation \eqref{ELeqn}; hence, it is a critical point of $\cals\cale^T_2$ under variations of fixed volume. 
\end{remark}

\begin{proposition}\label{nonconstsoln}
Non-extremal critical points of $\cals\cale$ exist. 
\end{proposition}

\begin{proof}
The critical points satisfying $\pi_gs^T_g=\bar{\bfs}^T_g$ need not be extremal. Indeed when $c_1(\cald)=0$ (or a torsion class) the results of Gauntlett, Martelli, Sparks, and Yau \cite{GMSY06,MaSpYau06} (cf. Theorem 11.3.14 of \cite{BG05}) give absolute minimum of $\bfH_\xi$ with $\bfS_\xi>0$ that do not have a Sasaki metric of constant scalar curvature. More generally the results of \cite{BHL17} show that a global minimum of $\bfH_\xi$ exists whether $\xi$ is extremal or not, and these are given by the condition $\pi_gs^T_g=\bar{\bfs}^T_g$.
\end{proof}

\begin{remark}\label{strnonextexample}
There are many explicit examples of non extremal critical points in Proposition \ref{nonconstsoln} which include homotopy spheres. See for example the Tables in \cite{BovCo16}. For all of these examples there are no extremal Sasaki metrics in the entire Sasaki cone.
\end{remark}

Thus, we have
\begin{theorem}\label{isothm}
Consider the family of Sasakian structures $\gt^+(\cald,J)$ and let $\xi\in\gt^+(\cald,J)$ have constant scalar curvature $s^T_g\neq 0$. Suppose also that there is no eigenfunction of the Laplacian $\Delta$ in $\calk^\xi$ with eigenvalue $\frac{s^T_g}{2n+1}$. Then its ray is isolated in the space of rays with vanishing transversal Futaki invariant.
\end{theorem}

\begin{proof} Recall that a Sasaki structure has vanishing transversal Futaki invariant if and only if $\pi_gs^T_g$ is a constant.  
The map $\Pi:\gt^+(\cald,J)\longrightarrow C^\infty(M)^\bbt$ defined by $$\Pi(\xi)= \pi_gs^T_g$$ is a homogeneous map of degree $1$ defined on $\gt^+(\cald,J)$ a convex open subset of a finite dimensional affine space. Identifying the tangent space of $\gt^+(\cald,J)$ with the Lie algebra $\gt$ and the tangent space of the Fr\'echet manifold $C^\infty(M)^\bbt$ with itself, we see that if the differential $d_\xi\Pi : \gt \rightarrow  C^\infty(M)^\bbt$ is injective at a point $\xi$, the constant rank theorem implies that there exists an open neighborhood $U$ of $\xi$ such that the restriction $\Pi:U \rightarrow C^\infty(M)^\bbt$ is injective. Assume moreover, that $\Pi(\xi)$ is a constant, then the pre-image of the constants in $\Pi(U)$ coincides with one isolated ray in $U$. We show that $d_\xi\Pi$ is injective when $\xi$ is cscS. Now from Lemma \ref{varlem1} we have
\begin{equation}\label{varlem2}
d_\xi\Pi(\eta(\dot{\xi}))= \frac{d}{dt}\left(\pi_{g_{\xi_t}}s^T_{g_{\xi_t}}\right)_{|_{t=0}}  = -(2n+1)\Delta^g \eta(\dot{\xi}) - ns^T_g \eta(\dot{\xi})+ (n+1) \pi_gs^T_g\eta(\dot{\xi}). 
\end{equation}
So if $s^T_g$ is constant \eqref{varlem2} becomes
$$d_\xi\Pi(\eta(\dot{\xi})) = -(2n+1) \Delta^g \eta(\dot{\xi}) + s^T_g \eta(\dot{\xi}).$$
Therefore, if $\dot{\xi}$ is colinear to $\xi$, the assumption $s^T_g \neq 0$ implies that $d_\xi\Pi(\eta(\dot{\xi})) \neq 0$. Otherwise, $d_\xi\Pi(\eta(\dot{\xi})) = 0$ gives that $\eta(\dot{\xi})\in\calh^\xi$ is an eigenfunction of $\Delta^g$ with eigenvalue $\frac{s^T_g}{2n+1}$. That is, the hypothesis of the theorem guarantees that $d_\xi\Pi : \gt \rightarrow  C^\infty(M)^\bbt$ is injective. This concludes the proof.  
\end{proof}

\begin{remark}\label{isorem}
In the K\"ahler--Einstein case, the space of Killing potentials coincides with the eigenspace of the first (non-trivial) eigenvalue $\lambda_1^g$ of the Laplacian thanks to results of Matsushima \cite{Mat57b} (see Theorem 3.6.2 of \cite{Gau09b}) which also provides a lower bound (which would read in our notation $s^T_g/2n$) on the eigenvalues of the Laplacian. There is an analogous result in the Sasaki case and this is exactly what we used  (i.e $\lambda_1^g \geq s^T_g/2n > s^T_g/(2n+1)$) in the proof of \cite[Theorem 1.7]{BHLT15} to get a local convexity result in the Sasaki $\eta$--Einstein case. In the toric K\"ahler--Einstein case, the fact that torus invariant Killing potentials (i.e affine linear function on the moment polytope) are eigenfunctions of the same eigenvalue even characterizes K\"ahler--Einstein metrics see \cite[Proposition 1]{LegSenaDias15}. Therefore, it would be surprising that there would be no constant scalar curvature K\"ahler metrics having a Killing potential as eigenfunction of the Laplacian; what we can hope, however, is that if it does the eigenvalue is not as low as $s^T_g/(2n+1)$. 
\end{remark}

Proposition \ref{Zalgvar} and Theorem \ref{isothm} imply 
\begin{corollary}\label{isocor}
Suppose that every cscS metric in the family $\gt^+(\cald,J)$ satisfies the hypothesis of Theorem \ref{isothm}. Then the zero set of rays $Z$ of the Sasaki-Futaki invariant is finite. In particular, the number of cscS rays in $\gt^+(\cald,J)$ is finite.
\end{corollary}

\section{The Functionals $\bfH,\cals\cale$ on Lens Space Bundles over Riemann Surfaces}
In this section we study the extremal Sasakian structures on lens space bundles over Riemann surfaces of genus $\calg$ from the point of view of the functionals $\bfH(b)$ and $\cals\cale(b)$. This is a special case of what we have called an $S^3_\bfw$ join \cite{BoTo13,BoTo14a} which generally represents lens space bundles over a Hodge manifold written as $M\star_{\bfl}S^3_\bfw$ where $\bfw =(w_1,w_2)$, $\bfl=(l_1,l_2)$ and the components of both are relatively prime positive integers. In this case since the critical points all belong to a 2-dimensional subcone $\gt^+_\bfw$ of $\gt^+(\cald,J)$, the critical rays are all isolated in the subcone $\gt^+_\bfw$. So Problem \ref{prob1} is answered in the affirmative for the $\bfw$ cone of these Sasaki manifolds. Moreover, Theorem 1.1 of \cite{BoTo14a} says that Problem \ref{prob2} is also answered in the affirmative in this case. 

For ease of discussion we take $M$ to be the constant scalar curvature Sasaki structure on an $S^1$ bundle over a Riemann surface $\grS_\calg$ with its standard Fubini-Study metric. The functional $\bfH$ for these manifolds was studied in \cite{BHLT15}. Choosing a Reeb field in $\gt^+_\bfw$ with coordinates $(v_1,v_2)$ gives a ray $\gr_b\in\gR(\cald,J)$ where $b=v_2/v_1$. So from  \cite{BHLT15} we have
\begin{equation}\label{dN=1HE}
\bfH(b) = \frac{\left(b^2 l_1 w_1+2 b l_2(1-\calg)+l_1 w_2\right)^3}{b^2 (b w_1+w_2)^2}.
\end{equation}
It follows from  \cite{BHLT15} that the critical points of $H(b)$ correspond to the points where either
$H(b)=0$ or $F(b)=0$ where 
\begin{equation}\label{dN=1CSC}
F(b)=b^3 l_1 w_1^2+b^2 (\calg l_2 w_1+2 l_1 w_1 w_2- l_2 w_1)-b (\calg l_2 w_2+2  l_1 w_1 w_2- l_2 w_2)-l_1 w_2^2.
\end{equation}
Note that $F$ is just the Sasaki-Futaki invariant up to a constant multiple. So the zero set of $F$ is $Z$ which consists  precisely of those Sasaki metrics in the two dimensional subcone $\gt^+_\bfw\subset \gt^+(\cald,J)$ that satisfy $\pi_gs^T_g=\bar{\bfs}^T_g$. Indeed, in this case the vanishing of $F$ for a given ray, guarantees the existence of an admissible CSC metric\footnote{Without going into details here, this follows essentially from Section 2.4 of \cite{ACGT08} together with the discussion in Section 5.1 of \cite{BoTo14a}.}. Note that for $\calg>0$ we have the equality $\gt^+_\bfw=\gt^+(\cald,J)$. We also mention that it can be seen from Equation \eqref{dN=1HE} that the zeroes of $\bfH_\xi$ (also $\bfS_\xi$) for which the transverse scalar curvature $s^T_g$ does not vanish identically are all inflection points in this case.

We now consider the Sasaki energy functional $\cals\cale$.
A straightforward computation using results from \cite{ACGT08,BoTo13,BoTo14a,BHLT15} gives
\begin{equation}\label{calscaleeqn}
\cals\cale(b) = \frac{(g_1(b))^3}{b^4  (w_1 b + w_2) \left(b^2 w_1^2+4 b w_1 w_2+w_2^2\right)^3}
\end{equation}
where

\begin{eqnarray}
g_1(b) & = & b^5 l_1^2 w_1^3+3 b^4 l_1^2 w_1^2 w_2 + b^3w_1\left(  l_2^2(\calg-1)^2 +2(1- \calg) l_1 l_2 w_2 - l_1^2 w_1 w_2 \right).  \notag\\
            &+& 2 b^2 w_2\left(l_2^2(\calg-1)^2 +2(1- \calg) l_1 l_2 w_1 - l_1^2 w_1 w_2  \right) +3 b l_1^2 w_1 w_2^2+l_1^2 w_2^3.
\end{eqnarray} 

Now we observe that $\lim_{b\rightarrow 0} \cals\cale(b) = \lim_{b\rightarrow +\infty} \cals\cale(b) = +\infty$
and the derivative equals
$$
\cals\cale'(b) = \frac{4F(b) \left(g_1(b)\right)^2 g_2(b)}{b^5 (b w_1+w_2)^2 \left(b^2 w_1^2+4 b w_1 w_2+w_2^2\right)^4}
$$
where $F(b)$ is given by Equation \eqref{dN=1CSC} and $g_2(b)$ is given by
$$
\begin{array}{ccl}
g_2(b)& = & b^5 l_1 w_1^4\\
\\
&+ & b^4(- \calg l_2 w_1^3+7  l_1 w_1^3 w_2+l_2 w_1^3)\\
\\
&+& b^3(-2  \calg l_2 w_1^2 w_2+3  l_1 w_1^3 w_2+10  l_1 w_1^2 w_2^2+2  l_2 w_1^2 w_2)\\
\\
&+& b^2(-2  \calg l_2 w_1 w_2^2+10  l_1 w_1^2 w_2^2+3  l_1 w_1 w_2^3+2 l_2 w_1 w_2^2)\\
\\
&+& b(- \calg l_2 w_2^3+7  l_1 w_1 w_2^3+ l_2 w_2^3)\\
\\
&+ &l_1 w_2^4.
\end{array}
$$
Clearly $g_1(b)=$ is equivalent to $\cals\cale(b)=0$, which in turn corresponds to $\pi_g s^T_g$ being constantly zero. In the admissible case at hand this implies that the ray has a CSC admissible metric of vanishing constant transverse scalar curvature. In particular, $F(b)$ would vanish as well. We also mention that if $\pi_g s^T$ is constantly zero, then  we must have $\calg >1$.
Thus, critical points of $\cals\cale$ other than the ones coming from CSC rays, correspond to solutions to $g_2(b)=0$ for
$b>0$ with $b\neq w_2/w_1$. Note that for $\calg \leq 1$, there are no such solutions, but for $\calg \geq 2$ and $l_2$ sufficiently large, we do indeed get solutions to $g_2(b)=0$ with $b>0$. Moreover, as the examples below will show us, these extra critical points may or may not be extremal, and in general they do not arise the same way as in the case for $\bfH(b)$. Note also that if $\cals\cale(b)$ has precisely one critical point, then, due to the limit behavior, this has to be an absolute minimum of $\cals\cale(b)$.  

\subsection{Explicit Examples}
Next we give examples explicitly describing the critical points of $\cals\cale$ and $\bfH$ and their relationship. We know that the critical sets ${\rm Crit}$ of both functionals contain the zero set $Z$ of the Sasaki-Futaki invariant $\bfF$;  however, generally ${\rm Crit}(\bfH)$ and ${\rm Crit}(\cals\cale)$ are different. 

\begin{example}\label{arbgen}
Here we take $\bfl=(1,1)$ and $\bfw=(w_1,w_2)=(3,2)$ so that $M=M_3\star_{1,1}S^3_{3,2}$ is an $S^3$ bundle over a Riemann surface 
$\grS_\calg$. 
One can check that for $\calg\leq 3$, both $\bfH(b)$ and $\cals\cale(b)$ have only one critical point (a global minimum), located at the $b$-value corresponding to the CSC ray.
For $\calg\geq 4$, 
$$\bfH(b)=\frac{\left(3 b^2-2(\calg-1)b+2\right)^3}{b^2 (3 b+2)^2}$$
has three distinct critical points. However, by using Descartes rule of signs (giving the maximum possible number of positive real roots) 
on $g_2(b)$ for $\calg \leq 15$  supplemented by a manual check for $\calg = 16,17$, we see that
$$\cals\cale(b) = \frac{\left(27 b^5+54 b^4+6 b^3 \calg^2-36 b^3 \calg-6 b^3+4 b^2 \calg^2-32 b^2 \calg+4 b^2+36 b+8\right)^3}{b^4 (3 b+2) \left(9 b^2+24 b+4\right)^3}$$
has only one critical point for $\calg=0,1,\dots,17$. For $\calg \geq 18$ it can be checked that $g_2(b)$ has two distinct positive real roots, none of which correspond to the unique CSC ray. These zeroes are outside of the extremal range.

As a more specific example within this example, let us suppose that $\calg =4$. One may check that in this case every ray in the Sasaki cone has an admissible extremal Sasaki metric. The three critical points of $\bfH(b)$ are the inflection points at $b=\frac{1}{3} \left(3-\sqrt{3}\right)$ and  $b=\frac{1}{3} \left(\sqrt{3}+3\right)$, plus the location of the global minimum, $b \approx  0.81$, corresponding to the CSC ray. The last value is then also the location of the global minimum (and only extremum point) of $\cals\cale$.
\end{example}

\begin{example}\label{sehex}
The Sasaki manifold $M$ is a lens space bundle over a genus 2 Riemann surface $\grS_2$ which is represented as a Sasaki join $M=M_3\star_{1,101}S^3_{3,2}$ where $M_3$ is the constant sectional curvature $-1$ Sasaki structure on the primitive $S^1$ bundle over $\grS_2$. We refer to \cite{BoTo13} for this join construction and to \cite{BoTo14a,BHLT16} for the general description of Sasaki joins.  The Einstein-Hilbert functional for $M$ is treated in Example 5.8 of \cite{BHLT15}. Now the Sasaki cone $\gt^+(\cald,J)$ of $M$ is 2 dimensional represented by the first quadrant $v_1> 0,v_2>0$. Then setting $b=\frac{v_2}{v_1}$ in \cite{BHLT15} we showed that $\bfH(b)$ has three critical points located at $b\approx 0.099,0.685,67.3$. Moreover, the range of admissible extremal\footnote{Whether there exist extremal Sasaki metrics on $M$ in the same class $(\calS,\bar{J})$ that are not admissible is an open question at this time, although there are expected to be none.} structures in the Sasaki cone is the open interval $(b_1,b_2)$ with $b_1\approx0.295$ and $b_2\approx 1.455$. Only one of the critical points, $b\approx 0.685$ lies in this range and it is a local minimum with constant scalar curvature. The two remaining critical points lie outside of the admissible extremal range, and they are inflection points corresponding to $\bfS_b=0$. 

Now consider the Sasaki energy functional 
\begin{equation}\label{sasenergyex}
\cals\cale(b)=\frac{\left(27 b^5+54 b^4+58746 b^3+38356 b^2+36 b+8\right)^3}{b^4 (3 b+2) \left(9 b^2+24 b+4\right)^3}.
\end{equation}
This also has three critical points $b\approx 0.023, 0.685, 30.3$. Note also that the functional $\cals\cale^T_2$ has the same critical points. However, we now find that the cscS metric $b\approx 0.685$ represents a local maximum of $\cals\cale(b)$, while both $b \approx 0.023$ and $b \approx 30.3$ lie outside of the admissible extremal range
and both represent local minima with the latter being an absolute minimum.

Notice that, although $\bfH$ and $\cals\cale$ have the same number of critical points, two of them are inflections points of $\bfH$ with $\bfS_b=0$ and $s^T_g$ is not identically zero. They are not critical points of $\cals\cale$; nevertheless, $\cals\cale$ has two critical points that are not critical points of $\bfH$.
\end{example}

\begin{example}\label{sehex2}
As a variant of Example \ref{sehex}, we can calculate that for $\calg =2$, $l_1=1$, $l_2=19$, $w_1=3$, and $w_2=2$,
$$\cals\cale (b) = \frac{\left(27 b^5+54 b^4+1674 b^3+964 b^2+36 b+8\right)^3}{b^4 (3 b+2) \left(9 b^2+24 b+4\right)^3}.$$
Numeric computer calculations indicate that 
here $b \approx 0.4466$ and $b \approx 2.497$ are relative minima (with the latter being the absolute minimum) while
$b \approx 0.7335$ is a relative maximum corresponding to the CSC ray. 
We can also numerically check that here the range of admissible extremal structures in the Sasaki cone is the open interval $(b_1,b_2)$ with $b_1\approx 0.0472$ and $b_2\approx 5.93$. 
Thus the critical points of $\cals\cale$ all correspond to admissible extremal rays.
Comparatively, $b \approx 0.7335$ is the location of a local (global) minimum of $\bfH(b) = \frac{\left(3 b^2-38 b+2\right)^3}{b^2 (3 b+2)^2}$ and this function has inflection points at
$b\approx 0.05285$ and $b\approx 12.61$ corresponding to $\bfS_b=0$. Again as in Example \ref{sehex} $\bfH$ and $\cals\cale$ have the same number of critical points, but two of them play distinct roles in the two functionals. Moreover, as mentioned in Remark \ref{nocscrem} the critical points of $\cals\cale$ are all extremal.
 \end{example}

\begin{example} We consider a similar lens space bundle but now over a Riemann surface of genus $\calg =0$ in which case $M$ is an $S^3$ bundle over $S^2$. As before we put $l_1=1$, $l_2=101$, $w_1=3$, and $w_2=2$ which implies that $M$ is the non-trivial $S^3$ bundle over $S^2$. In this case the admissible extremal range is the entire first quadrant $v_1>0, v_2>0$ which is the so-called $\bfw$ subcone $\gt^+_\bfw$ of the 3 dimensional Sasaki cone $\gt^+(\cald,J)$. Restricted to $\gt^+_\bfw$ both $\bfH$ and $\cals\cale$ have precisely three critical points. For $\cals\cale$ we have
$$\cals\cale(b) =\frac{ (8 + 36 b + 43204 b^2 + 63594 b^3 + 54 b^4 + 
  27 b^5)^3}{(b^4 (2 + 3 b) (4 + 24 b + 9 b^2)^3)}.$$
The $b$ values corresponding to CSC rays are the only critical points of $\cals\cale$. In this case we get multiple cscS rays corresponding to $b\approx 0.022$, $b\approx 0.644$, and $b\approx 31.67$. The value in the middle corresponds to a local maximum whereas the other two values are locations of relative minima. Note that the existence of multiple cscS rays were first discovered in \cite{Leg10} for precisely these types of Sasaki manifolds. One can check that 
$$\bfH(b)=\frac{(2+202b+3b^2)^3}{b^2(3b+2)}$$ 
and that $\bfS_b>0$, so all critical points of $\bfH(b)$ are cscS rays. Thus, in this case the critical points of $\bfH$ and $\cals\cale$ coincide when restricted to the $\bfw$ subcone of $\gt^+(\cald,J)$. The numerators of the differentials of $\bfH$ and $\cals\cale$ have common factors.
\end{example}

\def\cprime{$'$} \def\cprime{$'$} \def\cprime{$'$} \def\cprime{$'$}
  \def\cprime{$'$} \def\cprime{$'$} \def\cprime{$'$} \def\cprime{$'$}
  \def\cdprime{$''$} \def\cprime{$'$} \def\cprime{$'$} \def\cprime{$'$}
  \def\cprime{$'$}
\providecommand{\bysame}{\leavevmode\hbox to3em{\hrulefill}\thinspace}
\providecommand{\MR}{\relax\ifhmode\unskip\space\fi MR }
\providecommand{\MRhref}[2]{%
  \href{http://www.ams.org/mathscinet-getitem?mr=#1}{#2}
}
\providecommand{\href}[2]{#2}


\begin{thebibliography}{ACGTF08}

\bibitem[ACG06]{ApCaGa06}
Vestislav Apostolov, David M.~J. Calderbank, and Paul Gauduchon,
  \emph{Hamiltonian 2-forms in {K}\"ahler geometry. {I}. {G}eneral theory}, J.
  Differential Geom. \textbf{73} (2006), no.~3, 359--412. \MR{2228318
  (2007b:53149)}

\bibitem[ACGTF04]{ACGT04}
V.~Apostolov, D.~M.~J. Calderbank, P.~Gauduchon, and C.~W.
  T{\o}nnesen-Friedman, \emph{Hamiltonian $2$-forms in {K}\"ahler geometry.
  {II}. {G}lobal classification}, J. Differential Geom. \textbf{68} (2004),
  no.~2, 277--345. \MR{2144249}

\bibitem[ACGTF08]{ACGT08}
Vestislav Apostolov, David M.~J. Calderbank, Paul Gauduchon, and Christina~W.
  T{\o}nnesen-Friedman, \emph{Hamiltonian 2-forms in {K}\"ahler geometry.
  {III}. {E}xtremal metrics and stability}, Invent. Math. \textbf{173} (2008),
  no.~3, 547--601. \MR{MR2425136 (2009m:32043)}

\bibitem[AH06]{AlHa06}
Klaus Altmann and J\"urgen Hausen, \emph{Polyhedral divisors and algebraic
  torus actions}, Math. Ann. \textbf{334} (2006), no.~3, 557--607. \MR{2207875}

\bibitem[BCR98]{BoCoRo98}
Jacek Bochnak, Michel Coste, and Marie-Fran\c{c}oise Roy, \emph{Real algebraic
  geometry}, Ergebnisse der Mathematik und ihrer Grenzgebiete (3) [Results in
  Mathematics and Related Areas (3)], vol.~36, Springer-Verlag, Berlin, 1998,
  Translated from the 1987 French original, Revised by the authors.
  \MR{1659509}

\bibitem[BCTF19]{BoCaTo17}
Charles~P. Boyer, David M.~J. Calderbank, and Christina~W.
  T{\o}nnesen-Friedman, \emph{The {K}\"{a}hler geometry of {B}ott manifolds},
  Adv. Math. \textbf{350} (2019), 1--62. \MR{3945589}

\bibitem[BG00]{BG00b}
Charles~P. Boyer and Krzysztof Galicki, \emph{A note on toric contact
  geometry}, J. Geom. Phys. \textbf{35} (2000), no.~4, 288--298. \MR{MR1780757
  (2001h:53124)}

\bibitem[BG08]{BG05}
\bysame, \emph{Sasakian geometry}, Oxford Mathematical Monographs, Oxford
  University Press, Oxford, 2008. \MR{MR2382957 (2009c:53058)}

\bibitem[BGS08]{BGS06}
Charles~P. Boyer, Krzysztof Galicki, and Santiago~R. Simanca, \emph{Canonical
  {S}asakian metrics}, Commun. Math. Phys. \textbf{279} (2008), no.~3,
  705--733. \MR{MR2386725}

\bibitem[BGS09]{BGS07b}
\bysame, \emph{The {S}asaki cone and extremal {S}asakian metrics}, Riemannian
  topology and geometric structures on manifolds, Progr. Math., vol. 271,
  Birkh\"auser Boston, Boston, MA, 2009, pp.~263--290. \MR{MR2494176}

\bibitem[BHL18]{BHL17}
Charles Boyer, Hongnian Huang, and Eveline Legendre, \emph{An application of
  the {D}uistermaat--{H}eckman theorem and its extensions in {S}asaki
  geometry}, Geom. Topol. \textbf{22} (2018), no.~7, 4205--4234. \MR{3890775}

\bibitem[BHLTF17]{BHLT15}
Charles~P. Boyer, Hongnian Huang, Eveline Legendre, and Christina~W.
  T{\o}nnesen-Friedman, \emph{The {E}instein-{H}ilbert functional and the
  {S}asaki-{F}utaki invariant}, Int. Math. Res. Not. IMRN (2017), no.~7,
  1942--1974. \MR{3658189}

\bibitem[BHLTF18]{BHLT16}
\bysame, \emph{Reducibility in {S}asakian geometry}, Trans. Amer. Math. Soc.
  \textbf{370} (2018), no.~10, 6825--6869. \MR{3841834}

\bibitem[BM93]{BM93}
A.~Banyaga and P.~Molino, \emph{G\'eom\'etrie des formes de contact
  compl\`etement int\'egrables de type toriques}, S\'eminaire Gaston Darboux de
  G\'eom\'etrie et Topologie Diff\'erentielle, 1991--1992 (Montpellier), Univ.
  Montpellier II, Montpellier, 1993, pp.~1--25. \MR{94e:53029}

\bibitem[Boy19]{Boy18}
Charles~P. Boyer, \emph{Contact structures of {S}asaki type and their
  associated moduli}, Complex Manifolds \textbf{6} (2019), no.~1, 1--30.
  \MR{3912431}

\bibitem[BTF14]{BoTo13}
Charles~P. Boyer and Christina~W. T{\o}nnesen-Friedman, \emph{Extremal
  {S}asakian geometry on {$S^3$}-bundles over {R}iemann surfaces}, Int. Math.
  Res. Not. IMRN (2014), no.~20, 5510--5562. \MR{3271180}

\bibitem[BTF16]{BoTo14a}
\bysame, \emph{The {S}asaki join, {H}amiltonian 2-forms, and constant scalar
  curvature}, J. Geom. Anal. \textbf{26} (2016), no.~2, 1023--1060.
  \MR{3472828}

\bibitem[BvC18]{BovCo16}
Charles~P. Boyer and Craig van Coevering, \emph{Relative {K}-stability and
  extremal {S}asaki metrics}, Math. Res. Lett. \textbf{25} (2018), no.~1,
  1--19.

\bibitem[Cal56]{Cal56}
E.~Calabi, \emph{The space of {K}\"ahler metrics}, Proceedings of the
  International Congress of Mathematicians, Vol. 1, 2 (Amsterdam 2, 1954)
  (Amsterdam), North-Holland, 1956, pp.~206--207.

\bibitem[Cal82]{Cal82}
\bysame, \emph{Extremal {K}\"ahler metrics}, Seminar on Differential Geometry,
  Ann. of Math. Stud., vol. 102, Princeton Univ. Press, Princeton, N.J., 1982,
  pp.~259--290. \MR{83i:53088}

\bibitem[CS18]{CoSz12}
Tristan~C. Collins and G\'abor Sz\'ekelyhidi, \emph{K-semistability for
  irregular {S}asakian manifolds}, J. Differential Geom. \textbf{109} (2018),
  no.~1, 81--109. \MR{3798716}

\bibitem[DS14]{DonSun14}
Simon Donaldson and Song Sun, \emph{Gromov-{H}ausdorff limits of {K}\"{a}hler
  manifolds and algebraic geometry}, Acta Math. \textbf{213} (2014), no.~1,
  63--106. \MR{3261011}

\bibitem[DS17]{DonSun17}
\bysame, \emph{Gromov-{H}ausdorff limits of {K}\"{a}hler manifolds and
  algebraic geometry, {II}}, J. Differential Geom. \textbf{107} (2017), no.~2,
  327--371. \MR{3707646}

\bibitem[FM95]{FuMa95}
Akito Futaki and Toshiki Mabuchi, \emph{Bilinear forms and extremal {K}\"ahler
  vector fields associated with {K}\"ahler classes}, Math. Ann. \textbf{301}
  (1995), no.~2, 199--210. \MR{1314584}

\bibitem[FM02]{FuMa02}
\bysame, \emph{Moment maps and symmetric multilinear forms associated with
  symplectic classes}, Asian J. Math. \textbf{6} (2002), no.~2, 349--371.
  \MR{1928634}

\bibitem[Gau10]{Gau09b}
Paul Gauduchon, \emph{Calabi's extremal {K}\"ahler metrics}, preliminary
  version, 2010.

\bibitem[GMSY07]{GMSY06}
Jerome~P. Gauntlett, Dario Martelli, James Sparks, and Shing-Tung Yau,
  \emph{Obstructions to the existence of {S}asaki-{E}instein metrics}, Comm.
  Math. Phys. \textbf{273} (2007), no.~3, 803--827. \MR{MR2318866
  (2008e:53070)}

\bibitem[Leg11]{Leg10}
Eveline Legendre, \emph{Existence and non-uniqueness of constant scalar
  curvature toric {S}asaki metrics}, Compos. Math. \textbf{147} (2011), no.~5,
  1613--1634. \MR{2834736}

\bibitem[Leg16]{Leg16}
\bysame, \emph{Toric {K}\"ahler-{E}instein metrics and convex compact
  polytopes}, J. Geom. Anal. \textbf{26} (2016), no.~1, 399--427. \MR{3441521}

\bibitem[Ler02]{Ler02a}
E.~Lerman, \emph{Contact toric manifolds}, J. Symplectic Geom. \textbf{1}
  (2002), no.~4, 785--828. \MR{2 039 164}

\bibitem[Ler04]{Ler04}
\bysame, \emph{Homotopy groups of {$K$}-contact toric manifolds}, Trans. Amer.
  Math. Soc. \textbf{356} (2004), no.~10, 4075--4083 (electronic). \MR{2 058
  839}

\bibitem[LSD18]{LegSenaDias15}
Eveline Legendre and Rosa Sena-Dias, \emph{Toric aspects of the first
  eigenvalue}, J. Geom. Anal. \textbf{28} (2018), no.~3, 2395--2421.

\bibitem[Mat57]{Mat57b}
Y.~Matsushima, \emph{Sur la structure du groupe d'hom\'eomorphismes analytiques
  d'une certaine vari\'et\'e k\"ahl\'erienne}, Nagoya Math. J. \textbf{11}
  (1957), 145--150. \MR{0094478 (20 \#995)}

\bibitem[MSY08]{MaSpYau06}
Dario Martelli, James Sparks, and Shing-Tung Yau, \emph{Sasaki-{E}instein
  manifolds and volume minimisation}, Comm. Math. Phys. \textbf{280} (2008),
  no.~3, 611--673. \MR{MR2399609 (2009d:53054)}

\bibitem[Sim00]{Sim00}
Santiago~R. Simanca, \emph{Strongly extremal {K}\"ahler metrics}, Ann. Global
  Anal. Geom. \textbf{18} (2000), no.~1, 29--46. \MR{MR1739523 (2001b:58023)}

\end{thebibliography}
\end{document}